\documentclass[11pt]{amsart}
\def\supp{\operatorname{supp}}
\def\Prob{\operatorname{Prob}}

\newtheorem{theorem}{Theorem}
\newtheorem{proposition}{Proposition}
\newtheorem{corollary}{Corollary}
\theoremstyle{remark}
\newtheorem{remark}{Remark}
\usepackage{epsfig} 
\usepackage{graphicx} 

\begin{document}
\baselineskip19pt

\title[Stochastic semigroups and cell cycle]{Applications of stochastic semigroups to cell cycle models}
\author[K. Pich\'or]{Katarzyna Pich\'or}
\address{K. Pich\'or, Institute of Mathematics,
University of Silesia, Bankowa 14, 40-007 Kato\-wi\-ce, Poland.}
\email{katarzyna.pichor@us.edu.pl}
\author[R. Rudnicki]{Ryszard Rudnicki}
\address{R. Rudnicki, Institute of Mathematics,
Polish Academy of Sciences, Bankowa 14, 40-007 Katowice, Poland.}
\email{rudnicki@us.edu.pl}

\thanks{This research was partially supported by
the  National Science Centre (Poland)
Grant No. 2017/27/B/ST1/00100}
\thanks{$^*$ Corresponding author: Ryszard Rudnicki}

\begin{abstract}
We consider a generational  and continuous-time two-phase model of the cell cycle. The first model is given by a stochastic operator, and the second by
a piecewise deterministic Markov process. In the second case we also introduce a stochastic semigroup which describes the  evolution of densities of the process.
We study long-time behaviour of these models.
In particular we prove  theorems on asymptotic stability and sweeping.
We also show the relations between both models.
\end{abstract}

\maketitle

\section{Introduction}
\label{intro}

The modeling of the cell cycle has a long history \cite{Pujo-Menjouet}. 
The core of the theory was formulated in the late  sixties~\cite{LR,Rubinow,Foerster}.   The  important  role in these models is played by  maturity  of cells.
A lot of new models appear in the eighties and we can divide them into two groups.
The first group contains  discrete-time models (generational models) which  describe the relation between the initial maturity of 
mother and  daughter  cells \cite{LM-cc,TH-cc,Tyrcha}. The second group is formed by continuous-time  models characterizing
the  time evolution of  distribution  of cell maturity \cite{CMR,M-R,Rotenberg} or cell size \cite{DHT,GH}.
The long-time behaviour  of  continuous-time models was studied  in \cite{BPR,Mac-Tyr,Pichor-MCM,RP-M,RP}.
Mathematical modelling of cell cycle is still  important and topical 
and new interesting models appear \cite{ACHQ,ACM,CDGV,FKB,LRB}.

In this paper we consider two-phase models of the cell cycle.
 The cell cycle is a series of events that
 take place in a cell leading to its replication.
Usually the cell cycle is divided
 into four phases \cite{ABLRRW,John,Morgan}.
The first one is the growth phase $G_1$ with
synthesis of various enzymes.
The duration of the phase $G_1$ is
highly variable even
for cells from one species. The DNA synthesis takes place
in the second phase $S$.
In the third  phase $G_2$
significant protein synthesis occurs,
which is required during the process of mitosis.
The last phase $M$
consists of nuclear division and cytoplasmic division.
Some models of cell cycle contains also a $G_0$ phase, 
where the cell has left cycle and has stopped dividing.
 From a mathematical point of view we can
simplify  the model by considering only two phases \cite{BT,SM,Tyrcha}.
The first phase is the growth phase $G_1$ and it is also called the 
\textit{resting phase}.
The second phase called the \textit{proliferating phase}
 consists of the phases $S$, $G_2$,
and $M$.
The duration  of the first phase is random variable  and of the second  phase is
almost constant. A cell can move from the resting phase to the proliferating phase
with some rate, which depends on the maturity of a cell.
Each cell is  characterized by its age  and maturity.
The maturity can be size, volume or contents of genetic material.

We investigate discrete and continuous-time models characterized by 
the same parameters. 
The discrete model is slightly extended version of that 
of Tyrcha \cite{Tyrcha}. In our model the growth of maturity in both phases is described by different functions. We include the derivation of this model to have the paper self-contained. A cell can move from the resting phase to the proliferating phase with the rate which depends on its maturity. Then it spends a fixed time $\tau$ in the second phase and divides into two cells which have the same maturity. The maturity of a daughter cell is determined by the maturity of the mother cell at the moment of division.
The mathematical model is given by a stochastic operator $P$  which describes the relation between densities  of maturity of new born cells in consecutive generations. 

The  continuous-time  model is given  by  a piecewise deterministic Markov process (PDMP), which describes consecutive descendants of a single cell. 
PDMPs are nowadays widely used in modeling of biological phenomena \cite{PR-Stein,RT-K-k}. Some PDMPs were applied to  describe  
statistical dynamics of recurrent biological events and could be used in cell cycle models \cite{LMT,Mac-Tyr}. The main problem 
with application of PDMPs to a two-phase model is 
to construct a stochastic process which has the Markov property.
In the first phase the state of a cell depends on its maturity, but 
the second phase has a constant length, so the state of the cell
depends on time of visit this phase.
The novelty of our model is that it consists a system of three differential equations which describes age, maturity, and phase of a cell.
We consider two different jumps. The first jump is a stochastic one,
when a cell enters the proliferating phase. 
The second one is deterministic at the moment of division. 
After the division of a cell, we consider time evolution of its daughter cell, etc. 
Since we include into the model age, maturity and phase of a cell, our process 
satisfies the Markov property.
The evolution  of densities of the PDMP corresponding to our model 
leads directly to a   continuous-time stochastic semigroup
$\{P(t)\}_{t\ge0}$. The densities of the PDMP satisfy 
a  system of partial differential equations with boundary conditions similar to that  in \cite{M-R}. 
It is interesting that the density of maturity satisfies a
first order partial differential equation
in which there is a  temporal retardation
as well as a nonlocal dependence in the maturation variable (\ref{posr-RF}).
This observation suggests that even rather complicated transport equations 
can be introduced by means of simple PDMPs and one can find 
numerical solutions of these equations by Monte Carlo methods.

We study long-time behaviour of the discrete-time semigroup $\{P^n\}_{n\in \mathbb N}$ and   the semigroup $\{P(t)\}_{t\ge0}$.
We are specially interested in asymptotic stability  and  
sweeping   \cite{LiM}. We recall that a stochastic   semigroup is sweeping from a set $A$ 
if 
\[
\lim_{t\to\infty}\int_A P(t)f\,d\mu =0
\]
 for each density $f$.  
We prove that  both semigroups satisfy  the Foguel alternative, i.e. 
they are asymptotically stable or sweeping from compact sets.
This result is based on
a decomposition theorem of a stochastic semigroup into asymptotically stable and
sweeping components \cite{PR-JMMA2016} (see also \cite{PR-SD2017} for substochastic semigroups).   
We  give some sufficient conditions for asymptotic stability and sweeping of the continuous-time stochastic semigroup. 
We also present an example such that  the operator $P$ is  asymptotically stable but the semigroup $\{P(t)\}_{t\ge0}$ is sweeping
from compact sets and explain this unexpected phenomenon.
It should be noted that stochastic semigroups are widely
applied to study asymptotic properties of biological models
(see \cite{Rudnicki-LN,RT-K-k} and references cited therein).

The organization of the paper is as follows. 
Section~\ref{s:ap} contains the definitions and results concerning 
asymptotic stability,  sweeping and  the Foguel
alternative for stochastic semigroups.
Biological and mathematical description of the cell cycle is presented in Section~\ref{s: biol-to-math}. 
In Section~\ref{s: d-t-model}  we investigate the discrete-time model and we prove 
that the stochastic operator $P$ related to this model satisfies the  
Foguel alternative (Theorem~\ref{AF-dccm}).  We also recall some sufficient conditions for asymptotic properties of $P$.
In Section~\ref{s: c-t-model} we introduce a continuous-time model as a PDMP
and we  show that the stochastic semigroup $\{P(t)\}_{t\ge 0}$ corresponding to this process 
satisfies the Foguel alternative.  In  Section~\ref{s: master} 
we show the relations between 
discrete-time and continuous-time models, which allow us to formulate some conditions for  asymptotic stability and sweeping of $\{P(t)\}_{t\ge 0}$. 
Finally, we compare asymptotic properties of both  models.

\section{Asymptotic properties of stochastic operators and semigroups}
\label{s:ap}
 Let a triple $(X,\Sigma,\mu)$ be a $\sigma$-finite measure space.
Denote by $D$ the subset of the space
$L^1=L^1(X,\Sigma,\mu)$
which contains all
densities
\[
D=\{f\in \, L^1\colon  \,f\ge 0,\,\, \|f\|=1\}.
\]
A linear operator $P\colon L^1\to L^1$
is called \textit{stochastic} 
if   $P(D)\subseteq  D$.
A family
$\{P(t)\}_{t\ge 0}$ of linear operators on $L^1$  is called a
\textit{stochastic  semigroup}  if it is a strongly
continuous semigroup and  all operators $P(t)$ are stochastic.
Now, we introduce some notions which characterize the
asymptotic behaviour of
iterates of stochastic operators 
$P^n$, $n=0,1,2,\dots$,
and stochastic semigroups 
$\{P(t)\}_{t\ge0}$.
The iterates of stochastic  operators 
form a discrete-time  
semigroup and we can use notation $P(t)=P^t$ for their  powers 
and we formulate most of definitions and results for both types of semigroups
without distinguishing them.
A stochastic semigroup $\{P(t)\}_{t\ge 0}$ is
{\it asymptotically stable} if  there exists a density $f^*$   such that
\begin{equation}
\label{d:as}
\lim _{t\to\infty}\|P(t)f-f^*\|=0 \quad \text{for}\quad f\in D.
\end{equation}
 From (\ref{d:as}) it follows immediately that  $f^*$ is {\it invariant\,} with respect to
$\{P(t)\}_{t\ge 0}$, i.e.  $P(t)f^*=f^*$ for
each $t\ge 0$.
A stochastic semigroup $\{P(t)\}_{t\ge 0}$ is 
called \textit{sweeping}
with respect to a set $B\in\Sigma$ if for every
 $f\in D$
\begin{equation*}
\lim_{t\to\infty}\int_B P(t)f(x)\,\mu(dx)=0.
\end{equation*}

Our aim is to find such conditions that  a stochastic semigroup $\{P(t)\}_{t\ge 0}$ 
is asymptotically stable or sweeping from all compact sets 
called the {\it Foguel alternative} \cite{LiM}.
We also want to find simple sufficient conditions for asymptotic stability and sweeping 
 for  operators and semigroups related to cell cycle models.

We assume additionally that  $X$ is a separable metric space and   $\Sigma=\mathcal B(X)$ is the $\sigma$-algebra of Borel subsets of $X$.
We  will consider a stochastic  semigroup $\{P(t)\}_{t\ge0}$
such that for each $t\ge 0$ we have
\begin{equation}
\label{w-eta0}
P(t)f(x)\ge \int_X q(t,x,y)f(y)\, \mu(dy)\quad \textrm{for $f\in D$},
\end{equation}
where $q(t,\cdot,\cdot)\colon X\times X\to [0,\infty)$ is a  measurable function
and the following condition holds:

\noindent (K) for every $y_0\in X$ there exist  an $\varepsilon >0$,  a $t>0$,
and a measurable function
$\eta\ge 0$ such that $\int \eta(x)\, \mu(dx)>0$ and
\begin{equation}
\label{w-eta2}
q(t,x,y)\ge \eta(x)\mathbf 1_{B(y_0,\varepsilon)}(y)\quad \textrm{for $x\in X$},
\end{equation}
where
$B(y_0,\varepsilon)=\{y\in X:\,\,\rho(y,y_0)<\varepsilon\}$.

\noindent We define condition (K) for a stochastic operator $P$ in the same way remembering the notation $P(t)=P^t$.
Condition (K) is satisfied if, for example, for every point $y\in X$ there exist
a $t>0$ and an $x\in X$ such that
the kernel   $q(t,\cdot,\cdot)$ is continuous in a neighbourhood of $(x,y)$  and
$q(t,x,y)>0$.

Now, we formulate  the Foguel alternative for some class of stochastic semigroups.
We need an auxiliary definition.
We say that a stochastic 
semigroup $\{P(t)\}_{t\ge 0}$
\textit{overlaps supports}
 if 
for every $f,g\in D\,$ there exists $t>0$ such that 
\[
\mu(\mathrm{supp}\, P(t)f\cap 
\mathrm{supp}\, P(t)g)>0.
\]
The
\textit{support}
of any measurable function $f$ is defined up to a set of measure zero by the
formula 
\[
\supp f = \{x\in X : f(x)\ne 0\}.
\]

\begin{proposition}
\label{prop:FA-SO}
Assume that $\{P(t)\}_{t\ge 0}$  satisfies {\rm (K)} and overlaps supports. Then
$\{P(t)\}_{t\ge 0}$ is sweeping or $\{P(t)\}_{t\ge 0}$  has an invariant density $f^*$ with a support $A$  and there exists
a positive linear functional $\alpha$  defined on $L^1(X,\Sigma,\mu)$
 such that 
\begin{itemize} 
 \item[(i)]  for every  $f\in L^1(X,\Sigma,\mu)$ we have
\begin{equation} 
\label{th2:ap2}
 \lim_{t\to \infty}\|\mathbf 1_{A} P(t)f-\alpha(f)f^*\|=0,
\end{equation}
\item[(ii)] if $Y=X\setminus A$, then for every  $f\in L^1(X,\Sigma,\mu)$ and for every compact set $F$ we have
\begin{equation} 
\label{th2:sw2}
\lim_{t\to \infty} \int\limits_{F\cap Y} P(t)f(x)\, \mu(dx)=0. 
 \end{equation}
 \end{itemize}
In particular, if $\{P(t)\}_{t\ge 0}$  has an invariant density $f^*$ with the support $A$ and  $X\setminus A$ is a subset of a compact set, 
then $\{P(t)\}_{t\ge 0}$ is asymptotically stable. 
\end{proposition}

The proof of Proposition~\ref{prop:FA-SO} is based on   theorems on asymptotic decomposition of stochastic operators \cite[Theorem 1]{PR-JMMA2016} and  stochastic semigroups 
  \cite[Theorem 2]{PR-JMMA2016}.
 \begin{theorem}
 \label{th:1}
Assume that $P$ satisfies {\rm (K)}.
Then there exist an at most countable set $J$, a family of disjoint measurable sets $\{A_j\}_{j\in J}$
such that $P^*\mathbf 1_{A_j}\ge\mathbf 1_{A_j}$ for $j\in J$,  a family 
$\{S_j\}_{j\in J}$ of 
 periodic stochastic operators on $L^1(A_j,\Sigma_{A_j},\mu)$ with $\Sigma_{A_j}=\{A\in \Sigma\colon A\subseteq A_j\}$ for $j\in J$,
and a family $\{R_j\}_{j\in J}$  of positive projections  $R_j\colon  L^1(X,\Sigma,\mu)\to L^1(A_j,\Sigma_{A_j},\mu)$
 such that 
\begin{itemize} 
 \item[(i)] for every $j\in J$ and for every  $f\in L^1(X,\Sigma,\mu)$ we have
\begin{equation} 
\label{th:ap}
 \lim_{n\to \infty}\|\mathbf 1_{A_j} P^n\!f-S^n_jR_jf\|=0,
\end{equation}
\item[(ii)] if $Y=X\setminus \bigcup\limits _{j\in J} A_j$, then  for every  $f\in L^1(X,\Sigma,\mu)$ and for every compact set $F$ we have
\begin{equation} 
\label{th:sw}
\lim_{n\to \infty} \int\limits_{F\cap Y} P^n\!f(x)\, \mu(dx)=0. 
 \end{equation}
 \end{itemize}
\end{theorem}

We recall that a  stochastic operator $S$ is called \textit{periodic}
 if there exists 
 a sequence of densities $h_1,\dots,h_k$ such that
\begin{equation}
\label{per-op1}
h_ih_j=0\textrm{ for $i\ne j$ and $h_1+\dots+h_k>0$ a.e.,} 
\end{equation} 
\begin{equation}
\label{per-op2}
Sh_i=h_{i+1} \textrm{ for $i\le k-1$ and $Sh_k=h_1$} 
\end{equation} 
and for every integrable function  $f$ we have $Sf= SQf$, where
\begin{equation}
\label{per-op3}
 \textrm{$Qf=\sum_{i=1}^k \alpha_i(f) h_i$,  
$\,\alpha_i(f)=\int_{B_i} f(x)\,\mu(dx)$,   $\,B_i=\supp h_i$}.
\end{equation} 
The operator $P$ can be restricted to the space $L^1(A_j,\Sigma_{A_j},\mu)$, i.e.  
if $\supp f\subseteq A_j$ then $\supp Pf\subseteq A_j$. The operators
 $S_j$  have the property $S_jh_i=Ph_i$ (see the proof of Lemma 9 
\cite{PR-JMMA2016}), which means that the functions $h_i$ are periodic densities  of 
$P$ such that $\supp P^nh_{i_1}\cap \supp P^nh_{i_2}=\emptyset$  for each $n$ and $i_1\ne i_2$.
The space $L^1(A_j,\Sigma_{A_j},\mu)$ can be canonically embedded in the space 
$L^1(X,\Sigma,\mu)$ and, therefore,  $R_j$ can be treated as the transformation from  $L^1(X,\Sigma,\mu)$ to itself.
In the statement of Theorem~\ref{th:1} we use the following definition of a projection. 
A linear transformation $T$ from a vector space to itself is  a \textit{projection} if  $T^2 = T$.

\begin{theorem}
 \label{th:2}
Let $\{P(t)\}_{t\ge0}$ be a stochastic semigroup which satisfies {\rm (K)}.
Then there exist an at most countable set $J$, a family of invariant densities
$\{f^*_j\}_{j\in J}$ with disjoint supports $\{A_j\}_{j\in J}$,  and a family
$\{\alpha_j\}_{j\in J}$
of positive linear functionals  defined on $L^1$
 such that
\begin{itemize}
 \item[(i)] for every $j\in J$ and for every  $f\in L^1$ we have
\begin{equation}
\label{pomoc:ogolny}
 \lim_{t\to \infty}\|\mathbf 1_{A_j} P(t)f-\alpha_j(f)f^*_j\|=0,
\end{equation}
\item[(ii)] if $Y=X\setminus \bigcup\limits _{j\in J} A_j$, then for every  $f\in L^1$ and for every compact set $F$ we have
\begin{equation}
\label{th2:sw}
\lim_{t\to \infty} \int\limits_{F\cap Y} P(t)f(x)\, \mu(dx)=0.
 \end{equation}
 \end{itemize}
\end{theorem}

In particular, we have 
\begin{corollary}
\label{col-sw}
Assume that a continuous-time stochastic semigroup $\{P(t)\}_{t\ge 0}$ satisfies 
condition (K) and has no invariant densities.
Then $\{P(t)\}_{t\ge 0}$ is sweeping from compact sets.
\end{corollary}

 \begin{proof}[Proof of Proposition~\ref{prop:FA-SO}] 
First, we consider the case of a stochastic operator. 
If  $P$ satisfies conditions (K) and overlaps supports, then  $J$ is an empty set or a singleton.
Indeed, if  $\supp  f \subseteq A_j$ then $\supp  Pf \subseteq A_j$, because  $P^*\mathbf 1_{A_j}\ge\mathbf 1_{A_j}$.
If $J$ has at least two elements, then   $\supp  P^nf\cap \supp  P^ng\subseteq A_1\cap A_2= \emptyset $
 for  $n\in\mathbb N$ and $f,g\in D$  such that  $\supp  f \subseteq A_1$ and $\supp  g \subseteq A_2$, which contradicts the assumption that 
 $P$ overlaps supports. If $J$ is a singleton, then the periodic operator  $S$ is in fact a projection on a one dimensional space because 
the overlaping property of $P$ excludes  the existence of two periodic densities
$h_1$ and $h_2$ such that $\supp P^nh_1\cap \supp P^nh_2=\emptyset$  for each $n$. Thus condition 
(\ref{th:ap}) takes the form  (\ref{th2:ap2}). If   $P$  has an invariant density $f^*$ with the support $A$ and $Y=X\setminus A$ is subset of a compact set, then 
from condition (\ref{th2:sw2}) it follows that
$\lim_{n\to \infty} \int_{Y} P^nf(x)\, \mu(dx)=0$. Since $P$ is a stochastic operator,  we have $\alpha (f)=1$ for any density $f$, and consequently,
 $P$ is asymptotically stable. 
 The proof for continuous-time stochastic semigroups is straightforward because we have at most one invariant density
and  conditions (\ref{th2:ap2}), (\ref{pomoc:ogolny})  coincide.
 \end{proof}

 If a continuous-time stochastic semigroup $\{P(t)\}_{t\ge 0}$  has a unique invariant density $f^*$
 and $f^*>0$, then according to Theorem~\ref{th:1} condition (K)  implies asymptotic stability of $\{P(t)\}_{t\ge 0}$.
We can strengthen considerably this conclusion replacing condition (K) by the following one.
 A substochastic semigroup $\{P(t)\}_{t\ge 0}$
is called {\it partially integral} if 
(\ref{w-eta0}) holds and 
\begin{equation*}
\int_X\int_X  q(t,x,y)\,\mu(dx)\,\mu(dy)>0
\end{equation*}
for some $t>0$.
\begin{theorem}[\cite{PR-jmaa2}]
\label{asym-th2}
Let $\{P(t)\}_{t\ge 0}$ be a continuous-time partially integral stochastic
semigroup. Assume that the  semigroup $\{P(t)\}_{t\ge 0}$ has
a unique invariant density $f^*$. If $f^*>0$ a.e., then the semigroup
$\{P(t)\}_{t\ge 0}$ is asymptotically stable.
\end{theorem} 
 
 \section{From the biological background to a mathematical description}
\label{s: biol-to-math}
 We start with a short biological description of the two phase-cell cycle models. The cell cycle is divided into the  resting and proliferating phase.
The duration  of the  resting phase is random variable $t_R$ which depends
on the maturity of a cell. 
The duration $t_P$ of the proliferating phase is
almost constant. Therefore, we assume that $t_P=\tau$, 
where $\tau$ is a positive constant.

The crucial role in the model  is played by 
a  parameter $m$  called \textit{maturity} which describes the state
of a cell in the cell cycle.
Without loss of generality we can assume 
that the minimum cell maturity $m_{\rm{min}}$ equals zero.

A cell can move from the resting phase to the proliferating phase 
with rate  $\varphi(m)$, i.e., a  cell with age $a$ and with maturity  $m$
enters the proliferation phase  during
a small time interval of length $\Delta t$
 with probability
$\varphi(m)\Delta t+o(\Delta t)$.

We assume that cells age with unitary velocity
and mature with a velocity $g_1( m)$ in the resting phase  
and with a velocity $g_2( m)$ in the proliferating phase.
The variable $a$ in the proliferating phase  is assumed to range from $a = 0$ 
at the point of commitment to $a = \tau$ at the point of cytokinesis. 
The maturity of the daughter cell  $\overline m$ is a function of the maturity
 of the mother cell $m$, i.e. $\overline m=h(m)$ (see Fig.~\ref{cell-pic}).
For example if $m$ is the volume of a cell,
 then $h(m) =m/2$. 
 
 \begin{figure}
\begin{center}
\begin{picture}(340,140)(20,10)
\includegraphics[viewport=115 554 492 703]{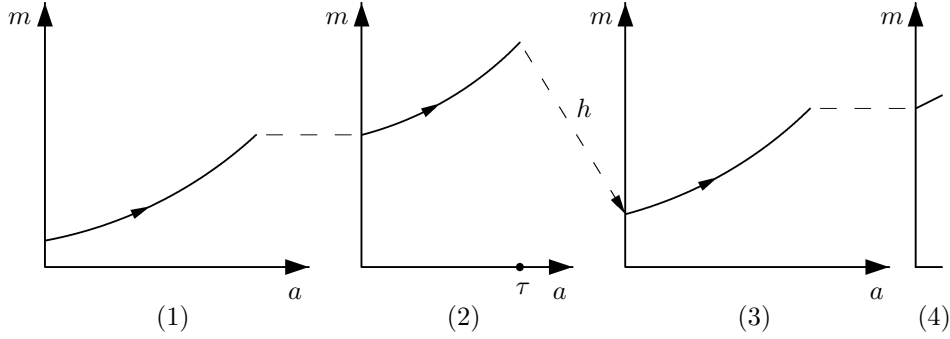}
\end{picture}
\end{center}
\caption{Evolution of maturity of a mother cell: (1) -- resting phase; 
(2) -- proliferating phase and 
a daughter cell: (3) -- resting  phase; (4)  --  proliferating phase.}
\label{cell-pic}
\end{figure}

Now we collect the assumptions concerning the model:
\vskip1mm

\noindent (M1) $\varphi$ is a continuous function such that
$\varphi (m)=0$ for $m\le m_P$ and
$\varphi (m)>0$ for $m> m_P$,
where $m_P>0$ is the minimum cell size when it can enter the proliferating phase,

\vskip1mm

\noindent (M2) $h\colon [m_P,\infty)\to [0,\infty)$ is  a $C^1$-function such that $h'(m)>0$,

\vskip1mm

\noindent (M3) $g_1\colon [0,\infty)\to (0,\infty)$
 and $g_2\colon [m_P,\infty)\to (0,\infty)$ are  
$C^1$ functions which increase sublinearly,

\vskip1mm

\noindent (M4) $\lim_{m\to\infty}
 \int \limits_{0}^{m} \dfrac{\varphi (r)}{g_1(r)}\,dr=\infty$.

Denote by $\pi_i(t,m_0)$ the solution of the equation  
\begin{equation}
m'(t)=g_i( m(t)),\quad i=1,2,
\end{equation}
with the initial condition $m(0)=m_0\ge 0$. 
From (M3) it follows that $\pi_i$ is a nonnegative and increasing function of both variables.
It is obvious that $h(\pi_2(\tau,m_P))=m_{\rm{min}}=0$.

Now, we introduce two auxiliary functions, which are used in both models. 
Let $\psi\colon [m_P,\infty)\to [0,\infty)$ be given by 
$\psi(m)=h(\pi_2(\tau,m))$. If $m$ is the maturity of the mother cell when
it enters the proliferating phase, then  $\psi(m)$ is the initial maturity of a daughter cell. From (M2) and (M3) it follows that 
$\psi'>0$. Moreover $\psi(m_P)=h(\pi_2(\tau,m_P))=0$.
Let $\lambda(m)=\psi^{-1}(m)=\pi_2(-\tau,h^{-1}(m))$.
Then $\lambda'>0$ and $\lambda(0)=m_P$.

\section{A discrete-time model}
\label{s: d-t-model}
Now we consider a discrete-time model \cite{Tyrcha}. This model  describes the relation between initial maturity of 
 mother and  daughter  cells. 
We assume that a new born cell has maturity $m_0$ and we want to find the distribution of maturity of the daughter cell. In order to do it
we first need to set down the distribution of $t_R$. 

Let $\Phi(t)$ be the cumulative distribution function of $t_R$, i.e.
 $\Phi(t)=\mathrm{Prob\,} (t_R\le t)$.
Then
\[
\mathrm{Prob\,}( t< t_R \le t+\Delta t \,|\, t_R> t)= \frac{\Phi(t+\Delta t) -\Phi(t)}{1-\Phi(t)}
= 
 \varphi (\pi_1(t,m_0))\Delta t+o(\Delta t).
 \]
From this equation we obtain
\[
\Phi'(t)=(1-\Phi(t)) \varphi (\pi_1(t,m_0))
\]
 and we get
\begin{equation}
\label{tyr1}
\Phi(t)=1-\exp\Big\{- \int_0^t\varphi (\pi_1(s,m_0)) \,ds\Big\}.
\end{equation}
Since $d\pi_1/ds=g_1(\pi_1(s,m_0))$ we obtain
\[
\int_0^t\varphi (\pi_1(s,m_0)) \,ds=\int_{m_0}^{\pi_1(t,m_0)}\frac{\varphi(m)}{g_1(m)}\,dm=
Q(\pi_1(t,m_0))- Q(m_0),
\]
where $Q(m)=\int \limits_{0}^{m} 
\dfrac{\varphi (r)}{g_1(r)}\,dr$.
Hence
\begin{equation}
\label{tyr2}
\Phi(t)=1-e^{ Q(m_0) -Q(\pi_1(t,m_0))}.
\end{equation}
According to (M4) $\lim_{m\to\infty}Q(m)=\infty$, which guaranties that 
each cell enters the proliferating phase with probability one.
From  (\ref{tyr2}) it follows that
\begin{equation}
\label{tyr3}
\begin{aligned}
\Phi'(t)
&
=\frac{d}{dt}(Q(\pi_1(t,m_0)))e^{Q(m_0)-Q(\pi_1(t,m_0))}
\\&
=\varphi(\pi_1(t,m_0))e^{Q(m_0)-Q(\pi_1(t,m_0))}.
\end{aligned}
\end{equation}

Since the random variable $\pi_1(t_R,m_0)$ is the maturity of the cell when it enters the proliferating phase,   
its maturity at the moment of division is given by $\pi_2(\tau,\pi_1(t_R,m_0))$. Finally
the maturity of the  daughter cell is given by the random variable $\xi=\psi(\pi_1(t_R,m_0))$.

In order to find the density of the random variable $\xi$  we determine the expectation 
of the random variable $\mathbb E(F(\xi))$, where $F$ is any bounded and continuous real function. 
We have
\[
\begin{aligned}
\mathbb E(F(\xi))&=\mathbb E(F(\psi(\pi_1(t_R,m_0))))=\int_0^{\infty}F(\psi(\pi_1(t,m_0)))\Phi'(t)\,dt\\
&=\int_0^{\infty}F(\psi(\pi_1(t,m_0)))\varphi(\pi_1(t,m_0))e^{Q(m_0)-Q(\pi_1(t,m_0))}\,dt\\
&=\int_{m_0}^{\infty}F(\psi(y))Q'(y)e^{Q(m_0)-Q(y)}\,dy\\
&=\int_{\lambda^{-1}(m_0)}^{\infty}F(m)\lambda'(m)Q'(\lambda(m))e^{Q(m_0)-Q(\lambda(m))}\,dm.
\end{aligned}
\]
Thus  the random variable  $\xi$
has the density
\[
\mathbf 1_{[\lambda^{-1}(m_0),\infty)} (m)\lambda'(m)Q'(\lambda(m))e^{Q(m_0)-Q(\lambda(m))}.
\]
Moreover, if we assume that  the distribution of the initial maturity of mother cells has a density $f$, then 
the initial maturity of the daughter cells has density 
\begin{equation}
\label{T-operator}
Pf(m)= \int_0^{\lambda(m)}\lambda'(m)Q'(\lambda(m))e^{Q(y)-Q(\lambda(m))}f(y)\,dy.
\end{equation}
Then $P$ is a stochastic operator on the space $L^1[0,\infty)$.

\begin{theorem}
\label{AF-dccm}
The operator $P$ satisfies the Foguel alternative, i.e. $P$ is asymptotically stable or sweeping from compact sets.
\end{theorem} 
\begin{proof}
The operator $P$ is of the form  
 \[
 Pf(m)=\int_0^{\infty}q(m,y)f(y)\,dy
\] 
with $q(m,y)= w(m) \mathbf 1_{[0,\lambda(m)]}(y)e^{Q(y)}$ and $w(m)=\lambda'(m)Q'(\lambda(m))e^{-Q(\lambda(m))}$.
Since $\lambda(0)=m_P$ and $\lambda'>0$,    
$\lambda(m)>m_P$ for $m>0$,
and according to (M1) we have $\varphi(\lambda(m))>0$ for $m>0$.
Thus
$Q'(\lambda(m))=\varphi(\lambda(m))/g_1(\lambda(m))>0$ for $m>0$. 
We also have
\[
\lambda'(m)=\frac{g_2(\lambda(m))}{g_2(h^{-1}(m))h'(h^{-1}(m))}>0
\]
for $m\ge 0$, which gives $w(m)>0$ for $m>0$. If we fix $y_0\ge 0$, then we find $m_0>0$ such that $\lambda(m_0) > y_0$.
Then  the function $\eta(m)= w(m)\mathbf 1_{[m_0,\infty)}(m)$  satisfies (\ref{w-eta2})
 for  sufficiently small  $\varepsilon>0$. Thus condition (K) is fulfilled. Now we fix a density $f$ and let $\bar y\ge 0$ be a point such  that
 $\int_{\bar y}^{\bar y+\varepsilon} f(y)\,dy >0$ for each $\varepsilon>0$.
Since $\lambda$ is an increasing function and $\lim_{m\to\infty}\lambda(m)=\infty$, there is $\bar m>0$ such that 
$\lambda(m)>\bar y$ for $m\ge \bar m$. From the definition of $P$ it follows that $Pf(m)>0$ if $\int_0^{\lambda(m)} f(y)\,dy>0$.
Hence $Pf(m)>0$ for $m\ge \bar m$. Since for
any two densities $f$ and $g$ there is an $\alpha>0$
such that  $Pf(m)>0$ and $Pg(m)>0$ for $m\ge \alpha$, the operator $P$ overlaps supports.
Moreover, if $P$ has an invariant density $f^*$ then 
$[0,\infty)\setminus \supp f^* \subseteq [0,c]$ for some $c>0$. According to Proposition~\ref{prop:FA-SO} the operator $P$ satisfies the
Foguel alternative. 
\end{proof}

Theorem~\ref{AF-dccm} does not establish when the operator $P$ is asymptotically stable or sweeping. 
Here we give some sufficient conditions for these properties.
\begin{proposition}
\label{asym-P}
Let $\alpha(m)=Q(\lambda(m))-Q(m)$.
The following conditions hold:

\noindent {\rm(a)}
 if $\liminf\limits_{m\to\infty} \alpha(m)>1$, 
then $P$ is asymptotically stable.

\noindent {\rm(b)}  if  $\alpha(m)\le1$ for
sufficiently large $m$, then $P$ is sweeping from each bounded interval,

\noindent {\rm(c)}  if $\inf \alpha(m)>-\infty$, then the operator $P$ is completely mixing, i.e.
\[
\lim_{n\to\infty}\|P^nf-P^ng\|=0\qquad
\text{for $f,g\in D$.}
\] 
\end{proposition}
These results were proved, respectively, (a) in 
\cite{GL}, (b) in \cite {LoR}, and (c) in~\cite{Ru-cm}.   

If the operator $P$ has an invariant density $f^*$, then we can find the stationary distribution of  age and maturity in both phases.
 From (\ref{tyr2}) it  follows that if a cell has the initial maturity $m_0$, then it will not have left the resting phase before age $a$ with probability
$e^{ Q(m_0) -Q(\pi_1(a,m_0))}$ and has maturity $\pi_1(a,m_0)$ at age  $a$.
Thus the probability that cell remains in the resting phase at age $a$ and has maturity $\le m$ at this age is given by the formula
\begin{equation}
\label{density-both-1p1}
\int_ 0^{\pi_1(-a,m)}f^*(m_0)e^{ Q(m_0) -Q(\pi_1(a,m_0))}\,dm_0.
\end{equation}
Denote by $\tilde f^*(a,m,i)$ the stationary density of the distribution of age and maturity in both phases.
Then from (\ref{density-both-1p1}) it follows that
\begin{equation}
\label{density-both-1p2}
\begin{aligned}
\tilde f^*(a,m,1) &=
c\frac{d}{dm}\int_ 0^{\pi_1(-a,m)}f^*(m_0)e^{ Q(m_0) -Q(\pi_1(a,m_0))}\,dm_0\\
&=c\frac{g_1(\pi_1(-a,m))}{g_1(m)}f^*(\pi_1(-a,m))e^{Q(\pi_1(-a,m))-Q(m)}
\end{aligned}
\end{equation}
for $m\ge \pi_1(a,0)$ and $\tilde f^*(a,m,1)=0$ for  $m< \pi_1(a,0)$,
where $c>0$ is a normalized constant. 
Integrating  (\ref{density-both-1p2}) over the age variable $a$ gives
\begin{equation}
\label{density-R-in}
\bar f^*(m,1)=\int_0^{\infty}\tilde f^*(a,m,1)\,da =
\frac{c}{g_1(m)}e^{-Q(m)}
\int_0^m e^{Q(x)}f^*(x)\,dx.
\end{equation}

In order to find $\tilde f^*(a,m,2)$ we need to find  the distribution of maturity at the beginning of proliferating phase.
We claim that the density of this distribution  is given by
$f_p^*(m)=\psi'(m)f^*(\psi(m))$. 
Indeed, if $\zeta$ is  a random variable having density $f^*$, then  the density of random variable
$\lambda (\zeta)$ coincides with  the  density of maturity at the beginning of proliferating phase. Thus 
\[
\Prob(\lambda (\zeta)\le m)=\Prob(\zeta\le  \psi(m))=\int_0^{\psi(m)}f^*(r)\,dr,
\]
which proves our claim.
Analogously to  (\ref{density-both-1p2}) we find that 
\begin{equation}
\label{density-both-2p2}
\begin{aligned}
\tilde f^*&(a,m,2) =c\frac{g_2(\pi_2(-a,m))}{g_2(m)}f^*_p(\pi_2(-a,m))\\
 &=c\frac{g_2(\pi_2(-a,m))}{g_2(m)}\psi'(\pi_2(-a,m))f^*(\psi(\pi_2(-a,m))),
\end{aligned}
\end{equation}
for $m\ge \pi_2(a,m_P)$ and $\tilde f^*(a,m,2)=0$ for  $m< \pi_2(a,m_P)$.
We have the same constant $c$  in the both formulas (\ref{density-both-1p2}) and (\ref{density-both-2p2}) 
because $\tilde f^*(\tau,m,2)=h'(m)\tilde f^*(0,h(m),1)$.
We can  find  the constant $c$ 
using the  formula 
\[
 \int_0^{\infty}\int_0^{\infty} \tilde f^*(a,m,1) \,da\,dm+ \int_{m_P}^{\infty}\int_0^{\tau}\tilde f^*(a,m,2) \,da\,dm =1.
\]
It is clear that the second integral equals $\tau$
and the first integral is the mean length  $T_R$ of the resting phase and 
\[
\begin{aligned}
T_R&=\int_0^{\infty}\int_0^{\infty} \tilde f^*(a,m,1) \,da\,dm\\
&=\int_0^{\infty}\int_{\pi_1(a,0)}^{\infty}
\frac{g_1(\pi_1(-a,m))}{g_1(m)}f^*(\pi_1(-a,m))e^{Q(\pi_1(-a,m))-Q(m)}\,dm\,da.
\end{aligned}
\]
Substituting $y=\pi_1(-a,m)$ and then $x=\pi_1(a,y)$  we obtain 
\begin{equation}
\label{wz:t_R}
\begin{aligned}
T_R&=\int_0^{\infty}\int_0^{\infty}
f^*(y)e^{Q(y)-Q(\pi_1(a,y))}\,dy\,da\\
&=\int_0^{\infty}\int_{y}^{\infty}\frac{1}{g_1(x)} e^{Q(y) -Q(x)}f^*(y)\,dx\, dy.
\end{aligned}
\end{equation}
Thus $c=1/(T_R+\tau)$ assuming that $T_R<\infty$.

\section{A continuous-time model}
\label{s: c-t-model}
Now we consider a continuous version of the model.
The cell cycle can be described as a piecewise deterministic Markov process.  
We consider a sequence of consecutive descendants
of a single cell. Let $s_n$ be a time when a
cell from the $n$-generation enters a resting phase and $t_n=s_n+\tau$
be a time of its division. 
If  $t_{n-1}\le  t<t_n$ 
then the state  $\boldsymbol \xi(t)=(a(t),m(t),i(t))$ of the $n$-th cell  
is described by age $a(t)$, maturity $m(t)$ and the index $i(t)$,
 where $i=1$ if  a cell is in the resting phase and $i=2$ if it is in the proliferating phase.
Random moments $t_0,s_1,t_1,s_2,t_2,\dots$ are called \textit{jump times}. 
Between jump times the  parameters change according to the following system 
of equations:
\begin{equation}
\left\{
\begin{aligned}
a'(t)&=1,\\
m'(t)&=g_{i(t)}(m(t)),\\
i'(t)&=0.
\end{aligned}
\right.
\end{equation} 
The process $\boldsymbol \xi(t)$  changes at jump points according 
to the following roles:
\[
a(s_n)=0,
\quad
m(s_n)=m(s_n^{-}),
\quad
i(s_n)=2,
\]
and
\[
a(t_n)=0,
\quad
m(t_n)=h(m(t_n^{-})) ,
\quad
i(t_n)=1.
\]
If $m(t_{n-1})=m_0$ then 
the  cumulative distribution function $\Phi$ of $s_n-t_{n-1}$ is 
given by (\ref{tyr2}).
Then $\boldsymbol \xi(t)$ is a time-homogeneous Markov process.
If the distribution of $\boldsymbol\xi(0)$ is given  by a density function $f(0,a,m,i)$, i.e.
a measurable function of $(a,m,i)$ such that 
\[
\Prob(\boldsymbol \xi(t)\in A\times{i})= \iint\limits_A f(0,a,m,i)\,da\,dm   
\]
for any  Borel set $A$ and $i=1,2$, then  
$\boldsymbol \xi(t)$ has a density  $f(t,a,m,i)$.

Having a time-homogeneous Markov process $\boldsymbol \xi(t)$ with the property that if the random variable $\boldsymbol \xi(0)$  has a density  $f_0$, then 
$\boldsymbol \xi(t)$  has a density  $f_t$,   we can define a stochastic semigroup $\{P(t)\}_{t\ge 0}$ corresponding to $\boldsymbol \xi(t)$
by $P(t)f_0=f_t$. The  proper choice of the space $X$ of values of the process $\boldsymbol \xi(t)$  plays an important role in investigations of  the process 
and the semigroup  $\{P(t)\}_{t\ge 0}$. We define
\[
X=\{(a,m,1)\colon m\ge \pi_1(a,0), \,\,a\ge 0\}\cup \{(a,m,2)\colon m\ge \pi_2(a,m_p), \,\,a\in [0,\tau]\}, 
\]
$\Sigma=\mathcal B(X)$ and $\mu$ is the product of the two-dimensional Lebesgue measure and the counting measure on the set $\{1,2\}$ (see Fig.~\ref{cell-pic2}). 
\begin{figure}
\begin{center}
\begin{picture}(340,150)(10,-70)  
\includegraphics[viewport=122 613 477 737]{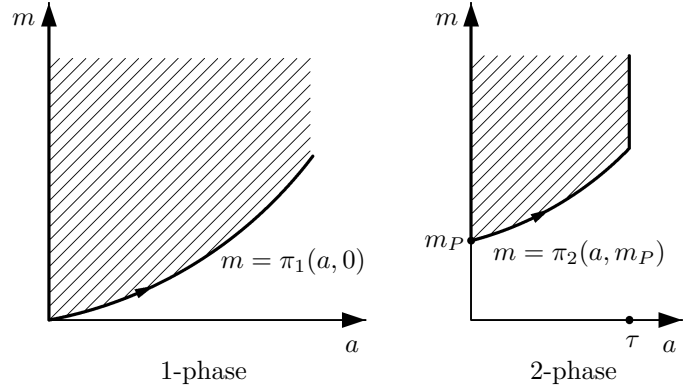}
\end{picture}
\end{center}
\caption{The set $X$}
\label{cell-pic2}
\end{figure}
Our aim is to check that the stochastic semigroup $\{P(t)\}_{t\ge 0}$ defined on $L^1(X,\Sigma,\mu)$ corresponding to the process $\boldsymbol \xi(t)$
satisfies the Foguel alternative and then to give some conditions for its asymptotic stability and sweeping.

We need two additional assumptions:
\begin{equation}
\label{EA1}
\psi(m)=h(\pi_2(\tau,m))<m \textrm{ for $m\ge m_P$}
\end{equation}
and 
\begin{equation}
\label{EA2}
h'(\pi_2(\tau,\bar m))g_2(\pi_2(\tau,\bar m))g_1(\bar m)\ne g_1(h(\pi_2(\tau,\bar m)))g_2(\bar m)
\end{equation}
for some $\bar m>m_P$.

Condition (\ref{EA1}) is not particularly restrictive because if $h(\pi_2(\tau,m_0))\ge m_0$ 
for some $m_0>m_P$, then  independently on the maturity of a cell, the probability that descended cells will have maturity $m< m_0$ goes to zero
as $t\to\infty$ and we can consider  a model with the minimal maturity $m_0$. 
If a mother cell has maturity $m>m_P$, then any number from the interval $(\psi(m),\infty)$ can be initial  maturity of
a daughter cell, any number from the interval $(\psi^2(m),\infty)$ can be initial maturity of a granddaughter cell if $\psi(m)>m_P$, etc. 
From condition (\ref{EA1}) it follows that for sufficiently large $n$ we have $\psi^n(m)\le m_P$. 
Since $\psi (m_P)=0$ we conclude that after a finite number of generations  the initial maturity of a descended cell can be any positive number $m$,
$m>\pi_1(a,0)$ can be the maturity of a  descended cell at age $a$ in the resting phase and 
$m>\pi_2(a,m_P)$ can be the maturity of a  descended cell at age $a$ in the proliferating phase.

Condition  (\ref{EA2}) seems to be technical but if  
\[
h'(\pi_2(\tau, m))g_2(\pi_2(\tau, m))g_1(m)= g_1(h(\pi_2(\tau, m)))g_2( m)
\]
for all $m\ge m_P$, then all descendants of a single  cell in the same generation have the same maturity at a given time $t$.
It means that the cell have \textit{synchronous growth} and we cannot expect the model is asymptotically stable.
 In particular if $g_1\equiv g_2$ and $h(m)=m/2$, then  (\ref{EA2}) reduces to $2g_2(m)\ne g_2(2m)$
 for some $m >\pi_2(\tau,m_P)$.
A similar condition appear in many papers concerning size-structured models 
\cite{BPR,DHT,GH,RP,RT-K-k}. 

Now we can formulate the Foguel alternative for semigroup $\{P(t)\}_{t\ge0}$ corresponding to the process $\boldsymbol \xi(t)$.
\begin{theorem}
\label{fa-semi}
The semigroup $\{P(t)\}_{t\ge0}$ satisfies the Foguel alternative, i.e. $\{P(t)\}_{t\ge0}$ is asymptotically stable or sweeping from compact sets.
\end{theorem} 
\begin{proof}
First we check condition (K). Let   $a_0>0$ and $m_0\in (\pi_1(a_0,0),\bar m)$ be the  age and  maturity of a cell at time $0$.
Define the functions $\theta_1(t_1,t_2)=t-t_1-t_2-2\tau$
 and  $\theta_2(t_1,t_2)=m_7$, where
 \[
 \begin{aligned}
&m_1=\pi_1(a_0+t_1,m_0),\quad m_2=\pi_2(\tau,m_1), \quad m_3=h(m_2),\\
& m_4=\pi_1(t_2,m_3), \quad m_5=\pi_2(\tau,m_4), \quad m_6=h(m_5),\\
& m_7=\pi_1(t-t_1-t_2-2\tau,m_6).
\end{aligned}
\] 
Then $\theta=(\theta_1,\theta_2)$ is  age and  maturity of a granddaughter cell at time $t>t_1+t_2+2\tau$.
It is easy to check that
\[
\theta'(t_1,t_2)=
\left[
\begin{array}{ll}
-1&-1\\
-g_1(m_7)+L_1L_2&-g_1(m_7)+L_1
\end{array} 
\right],
\]
where
\[
\begin{aligned}
L_1&=\frac{h'(m_5)g_1(m_7)g_2(m_5)g_1(m_4)}{g_1(m_6)g_2(m_4)},\\
L_2&=\frac{h'(m_2)g_2(m_2)g_1(m_1)}{g_1(m_3)g_2(m_1)}.
\end{aligned}
\]
Hence
\[
\det \theta'(t_1,t_2)=g_1(m_7)+L_1 L_2-g_1(m_7)-L_1=   L_1(L_2-1)
\]
and since $L_1\ne 0$,  the determinant of $\theta'(t_1,t_2)$ is different from zero  if and only if 
\begin{equation}
\label{w-K1}
h'(m_2)g_2(m_2)g_1(m_1)\ne g_1(m_3)g_2(m_1).
\end{equation}
Since $m_2=\pi_2(\tau,m_1)$ and $m_3=h(\pi_2(\tau,m_1))$ 
from (\ref{EA2}) it follows that condition (\ref{w-K1}) holds for $m_1$ sufficiently close to $\bar m$.  
Fix $t_1^0,t_2^0,t$  such that $\pi _1( a_0+t_1^0,m_0)=\bar m$, $t_2^0>0$, and $t>t_1^0+t_2^0+2\tau$.  
The times $t_1$ and $t_2$ are  random variables and we can find  densities of their distributions using formula (\ref{tyr3}). 
According to this formula there exist $\delta>0$ and $\varepsilon_1 >0$ such that 
their joint density $p(t_1,t_2)$   is bounded below by $\varepsilon_1 $ for 
$(t_1,t_2)\in (t_1^0-\delta,t_1^0+\delta)\times (t_2^0-\delta,t_2^0+\delta)$.
Since the age and  maturity of a granddaughter cell at time $t$ 
is given by $(a,m)=\theta(t_1,t_2)$
the function 
\[
\tilde p(a,m)=|\det (\theta^{-1})'(a,m)|p(\theta^{-1}(a,m))
\]
is the density of the distribution of $(a,m)$. Since $p(t_1,t_2)$ is bounded below by  $\varepsilon_1 >0$ and 
$\det \theta'(t^0_1,t^0_2)\ne 0$ 
we conclude that the density $\tilde p(a,m)$  is bounded below by some $\varepsilon_2>0$ for  $(a,m)$ from some neighbourhood  $V$ of  $\theta(t_1^0,t_2^0)$.
It means that
\[
q(t,(a,m,1),(a_0,m_0,1))\ge \varepsilon_2
\]
for $ (a,m)\in V$. We can also find a  neighbourhood  $U$ of  $(a_0,m_0,1)$ such that 
\[
q(t,x,y)\ge \varepsilon_2/2
\]
for  $x\in V$ and $y\in U$, 
i.e. 
(\ref{w-eta2})  holds for $y_0=(a_0,m_0,1)$.
Starting from any point $(\hat{a},\hat{m},\hat{\imath})\in X$ we can find a trajectory of the process $\boldsymbol \xi$  which joins it  with $(a_0,m_0,1)$.
Thus we can choose some neighbourhood $W$ of  $(\hat{a},\hat{m},\hat{\imath})$ and time $\bar t$ such that
if $z\in W$ than the process $\boldsymbol \xi$ starting from $z$ enters at time $\bar t$ the set $U$ with  probability $\ge  p_1$, where $p_1$ is a positive constant.
Hence 
\begin{equation}
\label{in-for-K} 
q(\bar t+t,x,z)\ge p_1\varepsilon_2/2
\end{equation}
for $z\in W$ and $x\in V$,
and consequently, condition (K) is fulfilled. 

Now we check that if $f^*$ is an invariant density for  the semigroup $\{P(t)\}_{t\ge 0}$ then $f^*>0$ a.e.
Let us take a point $y_0\in X$ such that  
the integral of $f^*$ over each neighbourhood of $y_0$ is positive.
Then from (\ref{in-for-K}) it follows that 
\begin{equation}
\label{nier:f>0}
\begin{aligned}
f^*(x)& =P(\bar t+t)f^*(x)\ge \int_X  q(\bar t +t,x,z)f^*(z)\,\mu(dz)\\
&\ge p_1\varepsilon_2/2 \int_W f^*(z)\,\mu(dz)>0,
\end{aligned}
\end{equation}
for $x\in V$.
Let  $\tilde x=(\tilde a,\tilde m_7,1)$, where 
\[
\tilde a=
\theta_1(t_1^0,t_2^0)= t-t_1^0-t_2^0-2\tau\quad
 \textrm{and}
\quad
 \tilde m_7=\theta_2(t_1^0,t_2^0).
\]
For $t_1=t_1^0$ we have $\tilde m_1=m_1(t_1^0)=\bar m $, $\tilde m_2=\pi_2(\tau,\bar m)$,
$\tilde m_3=\psi(\bar m)$. Let $\tilde m_4=m_4(t_1^0,t_2^0)=\pi_1(t_2^0,\tilde m_3)$. Since as $t^0_2$ we can choose any positive number, 
 $\tilde m_4$ can be any number from the interval $(\psi(\bar m),\infty)$. Let  $\tilde m_6=m_6(t_1^0,t_2^0) $. 
Then $\tilde m_6$ is any number from the interval $(c_1,\infty)$, where $c_1=0$ if $\psi(\bar m)\le m_P$ and $c_1=\psi^2(\bar m)$ if $\psi(\bar m)> m_P$.
Since $\tilde a= t-t_1^0-t_2^0-2\tau$,  $\tilde m_7=\pi_1(t-t_1^0-t_2^0-2\tau,\tilde m_6)$ and 
$t$ can be any number from the interval $(t_1^0+t_2^0+2\tau,\infty)$, $\tilde x$ can be any point from the set 
$A_{11}=\{(a,m,1)\colon \,a>0,\, m>\pi_1(a,c_1)\}$.  From $(\ref{nier:f>0})$ we obtain $f^*(a,m,1)>0$ for $(a,m)\in A_{11}$.
Since 
\[
f^*(0,m,2)=\varphi(m)\int_0^{\infty}f^*(a,m,1)\,da
\]
we have $f^*(0,m,2)>0$ for $m>c_2=\max(m_P,c_1)$.
 Hence $f^*(a,m,2)>0$ for $(a,m)\in A_{12}$, where  
$A_{12}=\{(a,m,2)\colon \,a\in [0,\tau],\, m>\pi_2(a,c_2)\}$.
Thus $f^*(0,m,1)>0$ for $m>c_3=\psi(c_2)$, and consequently
$f^*(a,m,1)>0$ for $(a,m)\in A_{21}$, where  
$A_{21}=\{(a,m,1)\colon \,a\ge 0,\, m>\pi_1(a,c_3)\}$
and 
$f^*(a,m,2)>0$ for $(a,m)\in A_{22}$, where  
$A_{22}=\{(a,m,2)\colon \,a\in [0,\tau],\, m>\pi_2(a,c_4)\}$, where $c_4=\max(m_P,c_3)$, etc.
Since  $c_{2k-1}=0$ for sufficiently large $k$,  we have
$f^*(a,m,1)>0$ for $A_1=\{(a,m,1)\colon \,a\ge 0,\, m>\pi_1(a,0)\}$
and  $f^*(a,m,2)>0$ for $A_2=\{(a,m,2)\colon \,a\ge 0,\, m>\pi_2(a,m_P)\}$. Hence 
$f^*>0$ a.e. on $X$.
Moreover, $f^*$ is the unique invariant density. Indeed, if a stochastic semigroup has two different invariant  densities $f_1$ and $f_2$,
then the function $(f_1-f_2)^+/ \|(f_1-f_2)^+ \|$ is also an invariant density and has the support smaller than $X$.

Therefore, if the  semigroup $\{P(t)\}_{t\ge 0}$ has 
an invariant density $f^*$, then this density is unique  and $f^*>0$ a.e. and
according to Theorem~\ref{asym-th2} this semigroup  is asymptotically stable.

If $\{P(t)\}_{t\ge0}$  has no invariant density, then according to Corollary~\ref{col-sw}
it is sweeping from compact sets.
\end{proof}

\section{Master equation}
\label{s: master} 
Theorem~\ref{fa-semi} guarantees that the  semigroup $\{P(t)\}_{t\ge0}$   
satisfies the Foguel alternative but if we want to check if this semigroup is asymptotically  stable or sweeping
we need to prove that it has or does not have an invariant density.
Time  evolution of densities can be described by some partial differential equations with boundary conditions
and knowing time independent solutions of this problem we can find an invariant density or check that such invariant density does not exist.    

Let $r(t,a,m):=f(t,a,m,1)$ and $p(t,a,m):=f(t,a,m,2)$. 
Then the functions $r$ and $p$ satisfy the following system of equations:
\begin{eqnarray}
\label{rownanie-r}
\frac{\partial r}{\partial t}
 +\frac{\partial r}{\partial a}
 +\frac{\partial(g_1(m)r)}{\partial m}
&=&-\varphi(m)r,\\
\frac{\partial p}{\partial t}
+\frac{\partial p}{\partial a}
+ \frac{\partial(g_2(m)p)}{\partial m}&=&
0,
\label{rownanie-p}
\end{eqnarray}
and the boundary conditions
\begin{equation}
\label{brzeg-r}
r(t,0,m)=k'(m)p(t,\tau,k(m)).
\end{equation}
\begin{equation}
p(t,0,m)=\varphi(m)\int_0^{\infty}r(t,a,m)\,da,
\label{brzeg-p}
\end{equation}
where $k=h^{-1}$. 

A similar system of equations was introduced in \cite{M-R}, where it described 
dynamics of a population of cells that are capable of simultaneous proliferation and maturation. 
That model includes,  among other things, mortality and does not lead directly to a stochastic semigroup.
In our case we replace one mother cell by one daughter cell which has allowed us to use a piecewise deterministic Markov process
in the model's description. 

 Let  $r(a,m)=\tilde f^*(a,m,1)$
and $p(a,m)=\tilde f^*(a,m,2)$, where  $\tilde f^*$ is given by 
(\ref{density-both-1p2}) 
and 
(\ref{density-both-2p2}).
It is not difficult to check that
$r(a,m)$ and $p(a,m)$ are solutions of (\ref{rownanie-r})--(\ref{rownanie-p}) with  boundary conditions 
(\ref{brzeg-r})--(\ref{brzeg-p}).
If 
\begin{equation}
\label{war:as-m2}
\int_0^{\infty} \int_0^{\infty}\tilde f^*(a,m,1)\,da\,dm+ \int_{m_P}^{\infty}\int_0^{\tau} \tilde f^*(a,m,2)\,da\,dm<\infty,
\end{equation}
then an invariant density exists and the semigroup $\{P(t)\}_{t\ge0}$ is asymptotically stable.
Condition (\ref{war:as-m2}) is equivalent to $T_R<\infty$, where $T_R$ is given by (\ref{wz:t_R}).
Moreover, one can check that $\tilde f^*$ is a unique, up to a  multiplicative constant,
 positive stationary solution  of (\ref{rownanie-r})--(\ref{brzeg-p}), which gives that if $t_R =\infty$ then 
 the semigroup has no stationary densities, and therefore it is sweeping from compact sets.
 We skip here the rigorous justification of  this statement.

The second integral  in $(\ref{war:as-m2})$  is finite, and therefore, in order to check 
if an invariant density exists it is enough to check that the first integral is finite.
In order to do it we investigate  the function 
 $R(t,m)$, which is the total number
of cells in the resting stage  with given maturity $m$ at time $t$, i.e.
\[
R(t,m)=\int_0^{\infty } r(t,a,m)\,da.
\]
Integrating equation (\ref{rownanie-r}) over the age variable $a$ and using
 boundary condition  (\ref{brzeg-r}) we obtain
\begin{equation}
\label{posr-R}
\frac{\partial R}{\partial t}+\frac{\partial(g_1R)}
{\partial m}= -\varphi(m)R + 
k'(m)p(t,\tau,k(m)).
\end{equation}
Applying the method  of characteristics to 
(\ref{rownanie-p})   and boundary condition (\ref{brzeg-p}) we find 
\[
\begin{aligned}
p(t,\tau,m)
&=
p(t-\tau,0,\pi_2(-\tau,m)) 
\frac{g_2(\pi_2(-\tau,m))}{g_2(m)}\\
&=\varphi(\pi_2(-\tau,m))R(t-\tau,\pi_2(-\tau,m))\frac{g_2(\pi_2(-\tau,m))}{g_2(m)}.
\end{aligned}
\]
Now equation (\ref{posr-R})  can be written in the following form
\[
\frac{\partial R}{\partial t}+\frac{\partial(g_1R)}
{\partial m}= -\varphi(m)R + 
 k'(m)\varphi(\lambda(m))
\frac{g_2(\lambda(m))}{g_2(k(m))}
R(t-\tau,\lambda(m)).
\]
We recall that $\lambda(m)=\pi_2(-\tau,k(m))$
and, in consequence, we finally obtain  
\begin{equation}
\label{posr-RF}
\frac{\partial R}{\partial t}+\frac{\partial(g_1R)}
{\partial m}= -\varphi(m)R + 
 \varphi(\lambda(m))\lambda'(m)
R(t-\tau,\lambda(m)).
\end{equation}
Now we are looking for a stationary solution of (\ref{posr-RF}). 
If  $R(m)$ satisfies (\ref{posr-RF}) then $R$ is a solution of the equation
\begin{equation}
\label{posr-RF-st}
(g_1R)'(m)= -\varphi(m)R(m) + 
 \varphi(\lambda(m))\lambda'(m)R(\lambda(m)).
\end{equation}
It is not surprising 
that if $\bar f^*(m,1)$ is given by (\ref{density-R-in}), then
$R(m)= \bar f^*(m,1)$ is a solution of (\ref{posr-RF-st}).
Moreover, if $R$ is a solution of  (\ref{posr-RF-st}) with $R(0)=0$, then the following formula holds:
\[
\varphi(m)R(m)=Q'(m)e^{-Q(m)}\int_{m_P}^{\lambda(m)}e^{Q(\lambda^{-1}(x))}\varphi(x)R(x)\,dx. 
\]
If we substitute $\tilde Q(m)=Q(\lambda^{-1}(m))$,  then  
$\tilde P(\varphi R)=\varphi R$,
where
\[
\tilde Pf(m)= \int_{m_P}^{\lambda(m)}\lambda'(m)\tilde Q'(\lambda(m))e^{\tilde Q(x)-\tilde Q(\lambda(m))}f(x)\,dx.
\]
Then 
\[
PU=U\tilde P, \quad Uf(m):=\lambda'(m)f(\lambda(m))
 \]
and $U$ is an isometric operator from $L^1[m_P,\infty)$ onto $L^1[0,\infty)$.  
Therefore, $\tilde P^n=U^{-1}P^nU$, and, in consequence, the operators $P$ and $\tilde P$ have the same asymptotic properties.  
Observe that $P$ has an invariant density $f^*$ if and only if  $\tilde f^*=U^{-1}f^*$ is an invariant density for $\tilde P$.

Let us assume that $P$ has an invariant density $f^*$ and $\varphi(m)\ge \varepsilon>0$ for sufficiently large $m$.
Since $R\varphi $ is a fixed point of $\tilde P$, 
we have $R\varphi=cU^{-1}f^*$ for some $c>0$.  
Hence, $\int_{m_P}^{\infty} R(m)\,dm<\infty$, which implies  that 
the  semigroup $\{P(t)\}_{t\ge0}$ is asymptotically stable.
According to Proposition~\ref{asym-P}, if 
$\liminf_{m\to\infty} Q(\lambda(m))-Q(m)>1$
and 
$\varphi(m)\ge \varepsilon>0$ for sufficiently large $m$, 
then the semigroup $\{P(t)\}_{t\ge0}$ is asymptotically stable.  

Now, we assume that $P$ has no invariant density and $\varphi$ is a bounded function.
Then $R$ cannot be an integrable function. Assume contrary to our claim, that $R$ is integrable. Then $R\varphi$ is an integrable function
and the operator $\tilde P$ has a positive fixed point. Hence $\tilde P$ and $P$ have invariant densities, a contradiction.
According to Proposition~\ref{asym-P}, if 
$Q(\lambda(m))-Q(m)\le 1$ for sufficiently large $m$ 
and $\varphi$ is bounded,
then the semigroup $\{P(t)\}_{t\ge0}$ is sweeping.

\begin{remark}
\label{r:rem-as-sw}
It can happen that the operator $P$ is asymptotically stable but the semigroup $\{P(t)\}_{t\ge0}$ is sweeping.
Indeed, if we choose $g_2$, $h$ and $\tau$ such that $\lambda(m)=m+2$ for $m\ge 0$
and we  choose $g_1$ and $\varphi$ such that $Q(m)=m$ for $m\ge 3$,   then
$\liminf\limits_{m\to\infty}( Q(\lambda(m))-Q(m))=2$ and the operator $P$ is asymptotically stable.
Let $f_*$ be an invariant density for $P$. The density $f^*$ depends only on $Q$ and $\lambda$, so we can 
choose $g_1$ and $\varphi$ such that  $\varphi(m)=g_1(m)=f^*(m-2)$ for $m\ge 3$. Then
$R(m)=cU^{-1}f^* (m)/\varphi(m)=c$. Consequently, $R$ is not integrable and the semigroup $\{P(t)\}_{t\ge0}$ is sweeping.
The explanation of this phenomenon is that in this example  the rate of entering the proliferating phase is very small for large $m$.
Then the mean length of the resting phase can be large and 
 more and more cells have arbitrary large maturity
as  $t\to\infty$.
 \end{remark}

\end{document}